\documentclass[12pt,letterpaper]{article}
\usepackage{amsmath}
\usepackage{amsfonts}
\usepackage{bm}
\usepackage{dsfont}
\usepackage{gensymb}
\usepackage{graphicx,psfrag,epsf}
\usepackage{geometry}
\usepackage{enumerate}
\usepackage{multirow}
\usepackage{color}
\usepackage{changes}
\usepackage{amsthm}
\usepackage{comment}

\usepackage{natbib}
\usepackage[
breaklinks,
colorlinks=true,
pagebackref=false,
pdfauthor={},%
pdftitle={},%
citecolor=blue,
linkcolor=blue,
urlcolor=blue,
filecolor=blue
]{hyperref}    

\newtheorem{theorem}{Theorem}
\newtheorem{lemma}[theorem]{Lemma}

\usepackage[labelfont = bf]{caption}
\usepackage{lineno}

\usepackage{bibentry}

\usepackage{url} 

\newcommand{\R}{\mathbb{ R}}

\newcommand{\mode}{\text{mode}}
\newcommand{\F}{\mathcal{F}}
\renewcommand{\P}{\mathcal{P}}

\begin{document}
\title{\bf Is the mode elicitable relative to unimodal distributions?}
 \author{Claudio Heinrich-Mertsching\thanks{Norwegian Computing Center Oslo, P.O. Box 114 Blindern, NO-0314 Oslo, Norway,
e-mail: \href{mailto:claudio@nr.no}{claudio@nr.no}} \and 
Tobias Fissler\thanks{Vienna University of Economics and Business (WU), Department of Finance, Accounting and Statistics, Welthandelsplatz 1, 1020 Vienna, Austria, 
	e-mail: \href{mailto:tobias.fissler@wu.ac.at}{tobias.fissler@wu.ac.at}}}
\maketitle
 \abstract{
 Statistical functionals are called elicitable if there exists a loss or scoring function under which the functional is the optimal point forecast in expectation. 
 While the mean and quantiles are elicitable, it has been shown in \cite{Heinrich2014} that the mode cannot be elicited if the true distribution can follow any Lebesgue density. We strengthen this result substantially, showing that the mode cannot be elicited if the true distribution can be any distribution with continuous Lebesgue density and unique local maximum.
 Likewise, the mode fails to be identifiable relative to this class.\\[-0.8em]
 }

\\[-0.8em]

\noindent
\textit{MSC2020 classes:}
62C99; 62F07; 62F10

\section{Introduction}
The mode of a probability density consists of the global maxima of this density. For a more general definition of the mode for probability distributions not admitting a Lebesgue or counting density, we refer to \cite{DearbornFrongillo2020}.
Together with the mean and the median, the mode is one the three main measures of central tendency in statistics \citep{DimitriadisPattonSchmidt2019}, which is reflected in many introductory textbooks to statistics as well as the \citeauthor{BoE}'s (\citeyear{BoE}) quarterly Inflation Report, which features all three point forecasts. 
The accuracy of point forecasts for the mean, median, mode, or any other functional $T$ on some class of probability densities $\F$ on $\R$ 
can be measured by loss or scoring functions. These are measurable functions $s\colon \R\times\R\to\R$ such that a forecast $x$ is penalized by the score $s(x,y)$ if $y$ materializes.
In order to encourage truthful forecasting, it has been advocated in the literature to use strictly consistent scoring functions for the target functional $T$ at hand \citep{Gneiting2011}.
For some subclass $\F'\subseteq \F$, the score $s$ is called strictly $\F'$-consistent for $T\colon\F\to\P(\R)$, where $\P(\R)$ is the power set of $\R$, if $\int |s(x,y)f(y)|\,\mathrm{d} y<\infty$
for all $x\in \R$ and $f\in\F'$ 
and if
\begin{equation}
\label{eq:cons}
\int s(t,y)f(y)\,\mathrm{d}y \le \int s(x,y)f(y)\,\mathrm{d}y
\end{equation}
for all $x\in\R$, $t\in T(f)$ and for all $f\in\F'$, and if equality in \eqref{eq:cons} implies that $x\in T(f)$.
If a functional admits a strictly $\F'$-consistent score, it is called \emph{elicitable} relative to $\F'$.
Besides facilitating meaningful forecast rankings which are exploited, e.g., in comparative backtests in finance \citep{NoldeZiegel2017}, strictly consistent scoring functions are the key tool in $M$-estimation and regression \citep{DimitriadisFisslerZiegel2020}, such as quantile or expectile regression \citep{KoenkerBasset1978, Koenker2005, NeweyPowell1987}.

The mean and the median are elicitable. Their most prominent scoring functions are the squared loss, $s(x,y) = (x-y)^2$ strictly consistent on the class of distributions with a second moment, and the absolute loss $s(x,y) = |x-y|$, which is strictly consistent on the class of distributions with a finite mean. 
On the class of counting densities, $\F_{\textrm{count}}$, the mode also admits a strictly consistent score in form of the zero-one loss, $s(x,y) = \mathds 1\{x\neq y\}$. While the mean and the median admit rich classes of strictly consistent scoring functions, the zero-one loss is essentially the only strictly $\F_{\textrm{count}}$-consistent score for the mode \citep{Gneiting2017}.
Relative to the class of all Lebesgue densities, however, \cite{Heinrich2014} showed that the mode fails to be elicitable.

Also other important functionals such as the variance or the expected shortfall have been found not to be elicitable relative to reasonably rich classes of probability distributions.
The proof strategy for these negative results was often simple and straightforward, exploiting the fact that any elicitable functional necessarily has convex level sets \citep{Osband1985}. This means that if two distributions have the same functional value, any mixture of these distributions has the same functional value. For corresponding versions for set-valued functionals, see \cite{FFHR2021}.
\citeauthor{Heinrich2014}'s (\citeyear{Heinrich2014}) result was historically the first to show that a functional with convex level sets fails to be elicitable.
This constitutes an exception to the equivalence result of \cite{SteinwartPasinETAL2014} who showed that, under weak regularity conditions, convex level sets are even sufficient for elicitability. The mode does not satisfy their conditions because it fails to be continuous, which also plays a crucial role in our proof.

The proof given in \cite{Heinrich2014} relies heavily on the fact that one can find Lebesgue densities with a high modal peak with arbitrarily small probability mass, while all the rest of the mass is concentrated elsewhere. 
So the core concept is closely related to the following fact:
\begin{center}
\begin{minipage}{0.7\textwidth}
{\it No matter how many random draws from a probability distribution we inspect, we can never be certain to have seen an observation near its mode.}
\end{minipage}
\end{center}

Essentially, it is always possible that the modal peak has a  mass  too small to be detected by the considered sample.
While this is certainly true from a probabilistic point of view, the relevance of this negative result for applied sciences is questionable. 
In a vast majority of applications, it is justified to assume that the data is distributed according to a Lebesgue density that has a unique local maximum, or is at least continuous. Neither of these cases is covered by the original proof given in \cite{Heinrich2014}.
In the meantime, the literature has made some progress regarding the elicitability question of the mode by showing that it is 
\emph{asymptotically} elicitable \cite{DimitriadisPattonSchmidt2019} under some restrictions, but generally even fails to be \emph{indirectly} elicitable \cite{DearbornFrongillo2020}.
The question as to whether the mode is elicitable relative to the relevant class of distributions with a continuous and unimodal Lebesgue density has been explicitly stated as an open problem (\textit{ibidem}).
In the present paper we answer this open question: Theorem \ref{maintheorem} shows that the mode remains not elicitable when the class of distributions is restricted to continuous densities with unique local maximum. 

It is impossible to generalize the proof of \cite{Heinrich2014} towards unimodal distributions, since for such distributions necessarily most of the mass is concentrated around the mode. The proof we present here follows an entirely different strategy and contains an alternative proof of \cite[Theorem 1]{Heinrich2014} as a special case.
Remarkably, our proof strategy can be applied to show that the mode also fails to be \emph{identifiable} relative to the class of continuous and unimodal densities; see Theorem \ref{identifiability-result}, and Section \ref{sec:Identifiability results} for precise definitions.
In applications, identifiability is crucial for forecast validation such as in calibration tests \citep{NoldeZiegel2017} as well as in $Z$-estimation or the (generalized) method of moments \citep{Hansen1982, NeweyMcFadden1994}.

In \cite{Heinrich2014} the term `strictly unimodal distributions' was used for densities with a unique {\it global} maximum, and is used differently in the present paper. We reserve the term `unimodal' for the stronger requirement of having only one \textit{local} maximum, which is more in line with the literature. Since we further consider continuous densities, the local maximum of a probability density is well-defined and we do not need to address ambiguities caused by densities differing on null sets.
Moreover, since the mode is a singleton on this class, we shall identify this singleton with its unique element when working with \eqref{eq:cons}.

\section{Elicitability results}
Denote by $\mathcal F_0$ the class of all strictly unimodal distributions with continuous Lebesgue density, i.e. of continuous densities with a unique local maximum.
The main result of our paper is the following theorem.

\begin{theorem}\label{maintheorem}
The mode is not elicitable relative to $\F_0$.
\end{theorem}

Theorem \ref{maintheorem} and the definition of strict consistency implies
 that the mode is also not elicitable relative to any superclass $\F\supseteq \F_0$ of distributions. In particular, the mode is not elicitable relative to the class of all continuous Lebesgue densities, neither is it elicitable relative to the class of all unimodal distributions.

Throughout the proof, we assume the existence of a strictly $\F_0$-consistent scoring function $s$ for the mode. 
The integrability condition that $\int|s(x,y)|f(y)\,\mathrm{d} y<\infty$ for all $x\in\R$ and for all $f\in\F_0$ implies that $s$ is \emph{locally $y$-integrable} in the sense that 
\[
\int_a^b |s(x,y)| \,\mathrm{d} y < \infty
\]
for all $a,b,x\in \R.$

The proof takes several steps to finally show that $s$ is necessarily constant along $y$-sections, which constitutes a contradiction to the strict $\F_0$-consistency. The key element of the proof is that we can pointwise approximate 
the density of the uniform distribution on $[a,b]$ (which is not in $\F_0$) by two sequences $f_n,g_n\in \F_0$ such that, for all $n\in\mathbb N$, $\mode(f_n) = x_1\neq x_2 = \mode(g_n)$,
where  $x_1,x_2\in [a,b]$. Intuitively speaking, we exploit the fact that the mode is not continuous with respect to pointwise convergence of densities.
This implies that 
\[\int_a^b s(x_1,y) \,\mathrm{d} y  = \int_a^b s(x_2,y) \,\mathrm{d} y,\]
see proof of Lemma \ref{firstlemma} for details.
This provides a powerful tool for deriving statements about $s$: By the Radon--Nikodym Theorem for signed measures, two measurable functions are equal (up to a null set) if their integrals over any interval $[a,b]$ are identical.
The described approximation of indicator function is the only requirement on the distribution class $\F_0$ used throughout the proof. Therefore, our proof shows that the mode is not elicitable relative to any class of distributions that allows approximations of indicator functions in this sense. In particular, our proof shows that the mode is not elicitable relative to the class of unimodal distributions with {\it smooth} densities, which is a slightly stronger statement than formulated in Theorem \ref{maintheorem} above.

We first show the following result.

\begin{lemma}\label{firstlemma}
Let $x_1\leq x_2\in\R$. Then
\(
s(x_1,y) = s(x_2,y)
\)
for Lebesgue almost all $y\not \in (x_1,x_2)$.
\end{lemma}
\begin{proof}
For $x_1 = x_2$ the statement is obviously true, and we may assume $x_1 < x_2$ throughout the proof.
Consider arbitrary but fixed $x_0 \leq x_1$ and $x_3\geq x_2$. We can find a uniformly bounded sequence $f_n\in\mathcal F_0$, all with mode in $x_1$, converging pointwise to the function $(x_3-x_0)^{-1} \mathds 1\{[x_0,x_3]\}$. Here and throughout the paper we denote by $\mathds 1$ the indicator function.
By the strict $\F_0$-consistency of $s$ we necessarily have
$\int s(x_1,y)f_n(y)\,\mathrm{d} y < \int s(x_2,y)f_n(y)\,\mathrm{d} y$ for all $n$, and therefore
\begin{align*}
(x_3-x_0)^{-1}\int_{x_0}^{x_3} s(x_1,y) \,\mathrm{d} y 
&= \lim_{n\to\infty}\int s(x_1,y)f_n(y)\,\mathrm{d} y\\
&\leq \lim_{n\to\infty}\int s(x_2,y)f_n(y)\,\mathrm{d} y
 = (x_3-x_0)^{-1}\int_{x_0}^{x_3} s(x_2,y) \,\mathrm{d} y.
\end{align*}
The convergence follows from the dominated convergence theorem since $f_n$ is uniformly bounded and $s$ is locally $y$-integrable. Here, we assumed without loss of generality that there is a bounded interval containing the support of all $f_n$, which allows us to find a dominating function.

However, if we replace $f_n$ by a sequence $g_n\in \F_0$ of densities, all with mode in $x_2$, that converges to the same function, we obtain the reversed inequality, and conclude that
\begin{align*}
\int_{x_0}^{x_3} s(x_1,y) \,\mathrm{d} y = \int_{x_0}^{x_3} s(x_2,y) \,\mathrm{d} y.
\end{align*}
Now, plugging in $x_0 = x_1$, we obtain
\begin{align*}
\int_{x_1}^{x_3} s(x_1,y) \,\mathrm{d} y = \int_{x_1}^{x_3} s(x_2,y) \,\mathrm{d} y,\quad \text{for all $x_3 \geq x_2$.}\end{align*}
In particular, we can subtract the equality for $x_3 = x_2$, and obtain
\begin{align*}
\int_{x_2}^{x_3} s(x_1,y) \,\mathrm{d} y = \int_{x_2}^{x_3} s(x_2,y) \,\mathrm{d} y,\quad \text{for all $x_3 \geq x_2$.}
\end{align*}
This implies the statement of the lemma for Lebesgue almost all $y>x_2$, by the Radon--Nikodym theorem for signed measures. The statement for $y < x_1$ follows by an analog argument.
\end{proof}

In preparation of the next result we recall Fubini's theorem for null sets, see \citet{vanDouwen1989}.
Recall that for $A\subseteq \R^2$ the $x$-section for $x\in\R$ is defined as 
\[A^x := \{y \in \R\,:\, \text{there is an } x\in\R \text{ such that }(x,y)\in A\}.\]
\begin{theorem}[Fubini's theorem for null sets]
\label{Fubini}
For any $A\subseteq \R^2$ the following statements are equivalent:
\begin{enumerate}
\item $A$ is a null set with respect to the Lebesgue measure on $\R^2$.
\item $A^x$ is a null set with respect to the Lebesgue measure on $\R$ for Lebesgue almost all $x\in\R$.
\end{enumerate}
\end{theorem}

We now use Lemma \ref{firstlemma} to derive the following statement:
\begin{lemma}
\label{secondlemma}
For all $a<b$ it holds that
\[ 
s(x,y) = s(a,y) \mathds 1\{x<y\} + s(b,y) \mathds 1\{x>y\}
\] 
for Lebesgue almost all $(x,y)\in \R\times [a,b]$. 
\end{lemma}
Note that Lemma \ref{secondlemma} is silent about the values of $s$ on the diagonal $\{x=y\}$, since this set is a null set in $\R^2$.
Before presenting a formal proof of Lemma \ref{secondlemma}, it is helpful to consider a geometrical visualisation of the proof strategy, for which we refer to Figure \ref{figure}.
\begin{figure}[t]
\centering
\begin{minipage}{0.4\textwidth}
\includegraphics[width = \textwidth]{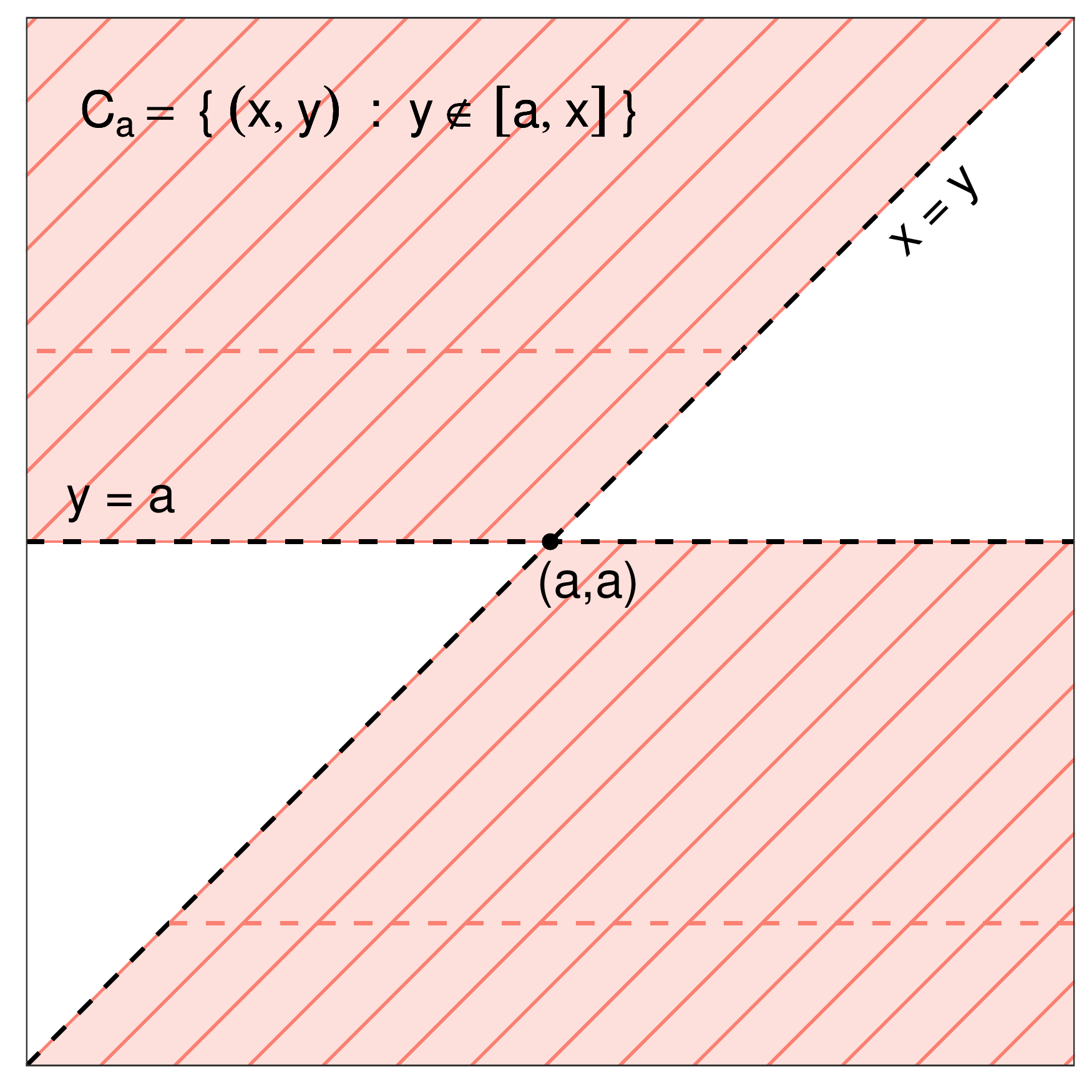}
\end{minipage}\hspace{2em}
\begin{minipage}{0.4\textwidth}
\includegraphics[width = \textwidth]{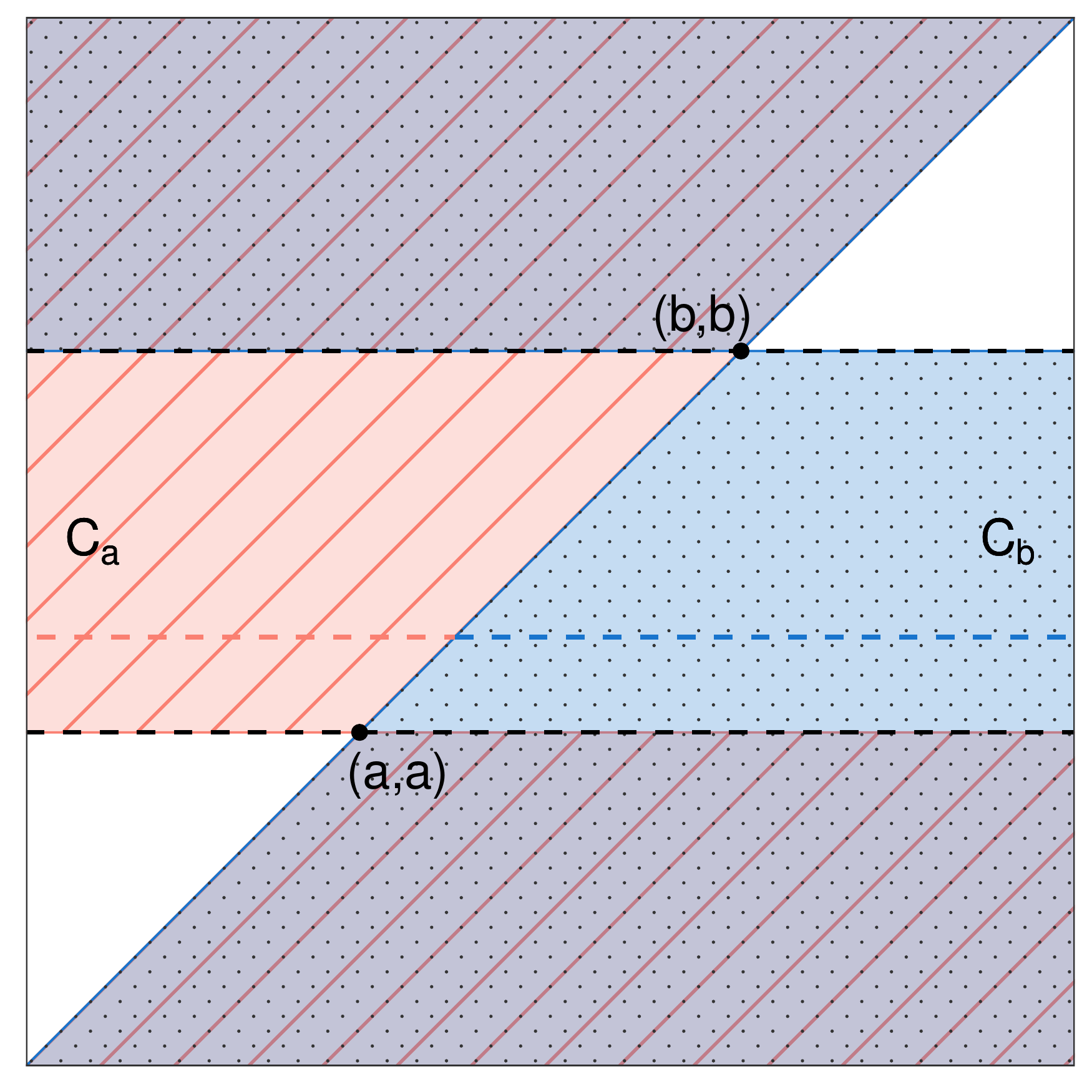}
\end{minipage}
\caption{ \small
Visualization of the proof of Lemma \ref{secondlemma}. The left panel shows the set $C_a$ hatched and in red. Lemma \ref{firstlemma}, applied with $x_1 = \min(a,x)$ and $x_2 = \max(a,x)$, implies that $s$ is constant along $y$-sections of $C_a$ (two examples of $y$-sections are shown by the dashed red lines). For Lemma \ref{secondlemma} we consider two such sets, $C_a$ and $C_b$ shown in the right panel, where $C_b$ is depicted dotted and in blue. 
Neglecting null sets, we obtain that for $y\in (a,b)$, the function $s$ takes at most two values on the $y$-section (example shown as dashed line), potentially changing values in $x=y$. 
\label{figure}}
\end{figure}
\begin{proof}[Proof of Lemma \ref{secondlemma}]
Let $a<b$.
For any $x\in \R$, consider the set $A^x_{a} := \{y: s(x,y) \neq s(a,y)\}$. By Lemma \ref{firstlemma}, applied with $x_1 = \min(a,x)$ and $x_2 = \max(a,x)$, the set $A^x_{a} \cap [a,x]^C$ is a null set (here and in general we follow the convention for intervals $[a,b] := [b,a]$, when $a>b$). 
This set is the $x$-section of the set $A_{a}\cap C_{a}$, where
\[A_{a} := \{(x,y) : s(x,y) \neq s(a,y)\},\qquad C_{a} := \{(x,y):y\in [a,x]^C\}.\]
See Figure \ref{figure} for a visualisation of the set $C_a$.
Consequently, since all $x$-sections of $A_{a} \cap C_{a}$ are null sets, $A_{a} \cap C_{a}$ is a null set itself, invoking Fubini's theorem for null sets.

Theorem \ref{Fubini} now implies that the set 
\[N_a := \{y\, :\, \text{the $y$-section of $A_{a} \cap C_{a}$ is not a null set}\}\]
is a null set in $\R$. Figure \ref{figure} shows examples of such $y$-sections. 
For $y_0 \in (a,b)$, the $y$-section of $A_{a}\cap C_{a}$ at $y_0$ is 
\begin{align}\label{zero-section}
A_{a}^{y_0}\cap C_{a}^{y_0} = 
\{x \,:\,x <y_0,\ s(x,y_0)\neq s(a,y_0)\}.
\end{align} 
Indeed, since $y_0>a$, we have $y_0\in [a,x]^C$ if and only if $x<y_0$. 

Repeating these arguments with $A_{a}\cap C_a$ replaced by $A_{b}\cap C_b$ shows that the $y$-sections of the latter are null sets as well, for all $y$ except on a null set $N_b$. For $y_1\in(a,b)$, we have that $y_1\in [b,x]^C$ if and only if $x>y_1.$
Therefore, the $y$-section of $A_{b}\cap C_{b}$ at $y_1$ is 
\begin{align}\label{one-section}
A_{b}^{y_1}\cap C_{b}^{y_1} = \{x\,:\,x > y_1,\ s(x,y_1) \neq s(b,y_1)\},
\end{align}
see the right panel in Figure \ref{figure} for an example.

Combining \eqref{zero-section} and \eqref{one-section} shows that, for all $y\in (a,b)\setminus (N_a\cup N_b)$, both the sets $A_a^y\cap C_a^y$ and $A_b^y\cap C_b^y$ are null sets. For such $y$ it holds that 
\[s(x,y) = s(a,y) \mathds 1\{x<y\} + s(b,y) \mathds 1\{x>y\},\]
for Lebesgue almost all $x\in \R$. Since this statement holds for Lebesgue almost all $y\in (a,b)$, the statement of Lemma \ref{secondlemma} follows from another application of Theorem \ref{Fubini}.
\end{proof}

Lemma \ref{secondlemma} essentially shows that the function $x\mapsto s(x,y)$ is piecewise constant for fixed $y$, with at most one jump at the diagonal $\{x = y\}$. In the following Lemma \ref{thirdlemma} we show that the function does not jump at the diagonal.

\begin{lemma}
\label{thirdlemma}
Let $a < b$. For Lebesgue almost all $y\in (a,b)$ it holds that $s(a,y) = s(b,y)$.
\end{lemma}
\begin{proof}
Lemma \ref{secondlemma} implies that the sets 
\[ 
B_x := \big\{y\in (a,b) : s(x,y) \neq s(a,y) \mathds 1\{x<y\} + s(b,y) \mathds 1\{x>y\}\big\}
\] 
are null sets for Lebesgue almost all $x\in \R$. Denote by $N$ the null set 
\[N = \{x\,:\, B_x \text{ has positive Lebesgue measure}\}.\]

Now, for $a < x_1 < x_2 < b$, with $x_1,x_2\not\in N,$ consider a sequence of 
uniformly bounded densities $f_n\in \F_0$ with support contained in $[a,b]$ and mode in $x_1$ that converges to $(x_2-x_1)^{-1}\mathds 1\{[x_1,x_2]\}$. 
By the strict $\F_0$-consistency of $s$, it holds for all $n$ that
\begin{align*}
0 &<\int_{a}^{b} \big[s(x_2,y) - s(x_1,y)\big]f_n(y) \,\mathrm{d} y\\
& =\int_{a}^{b} \big[s(a,y) \mathds 1\{x_2<y\} + s(b,y) \mathds 1\{x_2>y\} \\
&\phantom{=\int_{a}^{b} \big(} - s(a,y) \mathds 1\{x_1<y\} - s(b,y) \mathds 1\{x_1>y\}\big]f_n(y) \,\mathrm{d} y\\
& = \int_{x_1}^{x_2} \big[-s(a,y) + s(b,y)\big]f_n(y) \,\mathrm{d} y, 
\end{align*}
where the first equality holds since $x_1,x_2\not\in N$.
Letting $n\to\infty$, and applying the dominated convergence theorem, we obtain
\begin{align*}
0 \geq \int_{x_1}^{x_2} s(a,y) - s(b,y) \,\mathrm{d} y.
\end{align*}
However, by selecting a sequence $g_n\in\F_0$ similar to $f_n$ but with mode in $x_2$ rather than $x_1$, we obtain the reversed inequality
\begin{align*}
0 \leq \int_{x_1}^{x_2} s(a,y) - s(b,y) \,\mathrm{d} y.
\end{align*}
Consequently, we have $0 = \int_{x_1}^{x_2} s(a,y) - s(b,y) \,\mathrm{d} y$ for all $a<x_1<x_2<b$ with $x_1,x_2\not\in N$, and thus $s(a,y) = s(b,y)$, for Lebesgue almost all $y\in [a,b]$. 
\end{proof}

Combining Lemmas \ref{firstlemma} and \ref{thirdlemma} yields that, for any fixed $a\neq b$,  $s(a,y) = s(b,y)$ for Lebesgue almost all $y\in\R$. To conclude let $f\in \F_0$ be a density with mode in $a$. It holds that
\[
\int s(a,y)f(y)\,\mathrm{d} y = \int s(b,y)f(y)\,\mathrm{d} y.
\]
This contradicts the strict $\F_0$-consistency of $s$ and therefore completes the proof of Theorem \ref{maintheorem}. \qed

\section{Identifiability results}
\label{sec:Identifiability results}

A related question to the elicitability of a functional $T\colon\F\to\P(\R)$ is its \emph{identifiability}.
The functional $T$ is identifiable relative to $\F'\subseteq \F$ if there exists a strict $\F'$-identification function for it. 
That is a measurable function $v\colon\R\times\R\to\R$ such that $\int |v(x,y)f(y)|\,\mathrm{d}y<\infty$ for all $x\in\R$ and, for all $f\in\F'$, satisfying 
\[
\int v(x,y)f(y)\,\mathrm{d}y =0 \quad \Longleftrightarrow \quad x\in T(f)
\]
for all $x\in\R$, $f\in\F'$.
In Econometrics, identification functions are often called \emph{moment functions}.
Intuitively speaking, identification functions arise as derivatives of scoring functions, which can be made rigorous by \emph{Osband's principle} \citep{Osband1985, FisslerZiegel2016}.
\cite{SteinwartPasinETAL2014} not only show that convex level sets and elicitability are equivalent, subject to regularity conditions, but extend this result also to identifiability. 
Since the mode violates these regularity conditions, it is open whether the mode is identifiable relative to $\F_0$, the class of all strictly unimodal distributions with continuous Lebesgue density.
A slight modification of the proof of Lemma \ref{firstlemma} yields:

\begin{theorem}\label{identifiability-result}
The mode is not identifiable relative to $\F_0$.
\end{theorem} 
Theorem \ref{identifiability-result} generalizes \citet[Lemma 1]{DearbornFrongillo2020}, which establishes that the mode is not identifiable relative to the class of distributions with a unique \textit{global} maximum.
Again, Theorem \ref{identifiability-result} implies that the mode fails to be identifiable relative to any superclass $\F\supseteq \F_0$.

\begin{proof}[Proof of Theorem \ref{identifiability-result}]
The proof is largely analog to Lemma \ref{firstlemma}. Assume the existence of a strict $\F_0$-identification function $v$. Consider $a\leq x\leq b\in\R$ with $a<b$, and consider a sequence of uniformly bounded densities $f_n\in \F_0$, all with mode in $x$, converging pointwise to the scaled indicator function $(b-a)^{-1}\mathds 1\{[a,b]\}$. Since $\int v(x,y) f_n(y)\,\mathrm{d}y = 0$ for all $n$, an application of the dominated convergence theorem yields
\[
(b-a)^{-1}\int_a^b v(x,y)\,\mathrm{d}y = \lim_{n\to\infty} \int v(x,y) f_n(y)\,\mathrm{d}y = 0.
\]
Selecting $a = x$ and considering that this equality holds for any $b>x$, the Radon--Nikodym theorem shows that $v(x,y)=0$ for Lebesgue almost all $y>x$. Similarly, selecting $b=x$ shows that $v(x,y)=0$ for Lebesgue almost all $y<x$, and thus we have $v(x,y) = 0$ for almost all $y$. This yields a contradiction by choosing $f\in \F_0$ with mode in $x'\neq x$, since it holds that $\int v(x,y)f(y)\,\mathrm{d}y =0.$
\end{proof}

\section*{Acknowledgements}
The authors would like to thank Timo Dimitriadis for constructive comments which improved the content of the paper.
Claudio Heinrich-Mertsching is grateful to the Norwegian Computing Center for its financial support.

 \bibliographystyle{chicago}

\end{document}